\documentclass[a4paper,10pt]{article}
\usepackage{amsmath}
\usepackage{amssymb}
\usepackage{amscd}
\usepackage{amsthm}
\usepackage{graphicx}
\usepackage{pdfsync}
\newtheorem{theorem}{Theorem}
\newtheorem{lemma}[theorem]{Lemma}
\newtheorem{corollary}[theorem]{Corollary}

\newtheorem{result}[theorem]{Result}

\def\PG{\mathrm{PG}}

\def\F{\mathbb{F}_q}
\title{A proof of the linearity conjecture for $k$-blocking sets in $\PG(n,p^3)$, $p$ prime}

\author{M. Lavrauw \thanks{This author's research was supported by the Fund for Scientific Research - Flanders (FWO)}
\and L. Storme \and G. Van de Voorde $^*$}
\date{}
\begin{document}
\maketitle
\begin{abstract} In this paper, we show that a small minimal $k$-blocking set in $\PG(n,q^3)$, $q=p^h$, $h\geq 1$, $p$ prime, $p\geq 7$, intersecting every $(n-k)$-space in $1\pmod{q}$ points, is linear. As a corollary, this result shows that all small minimal $k$-blocking sets in $\PG(n,p^3)$, $p$ prime, $p\geq 7$, are $\mathbb{F}_p$-linear, proving the linearity conjecture (see \cite{sziklai}) in the case $\PG(n,p^3)$, $p$ prime, $p\geq 7$.
\end{abstract}

%%%%%%%%%%%%%%%%%%%%%%%%%%%%%%%%%%%%%%%%%%%%%%%%%%%

\section{Introduction and preliminaries}

%%%%%%%%%%%%%%%%%%%%%%%%%%%%%%%%%%%%%%%%%%%%%%%%%%%

Throughout this paper $q=p^h$, $p$ prime, $h\geq 1$ and
$\PG(n,q)$ denotes the $n$-dimensional projective space over the finite field $\F$ of order $q$.
A \emph{$k$-blocking set} $B$ in $\PG(n,q)$ is a set of points  such that any $(n-k)$-dimensional subspace intersects $B$. A $k$-blocking set $B$ is called {\em trivial} when a $k$-dimensional subspace is contained in $B$. If an $(n-k)$-dimensional space contains exactly one point of a $k$-blocking set $B$ in $\PG(n,q)$, it is called a {\em tangent $(n-k)$-space} to $B$. A $k$-blocking set $B$ is called {\em minimal} when no proper subset of $B$ is a $k$-blocking set. A $k$-blocking set $B$ is called {\em small} when $\vert B \vert <3(q^k+1)/2$.

Linear blocking sets were first introduced by Lunardon \cite{L1} and can be defined in several equivalent ways. 

In this paper, we follow the approach described in \cite{lavrauw2001}. In order to define a linear $k$-blocking set in this way, we introduce the notion of a Desarguesian spread.
Suppose $q=q_0^t$, with $t\geq 1$. By "field reduction", the points of $\PG(n,q)$ correspond to $(t-1)$-dimensional subspaces of $\PG((n+1)t-1,q_0)$, since a point of $\PG(n,q)$ is a $1$-dimensional vector space over ${\mathbb F}_q$, and so a $t$-dimensional vector space over ${\mathbb F}_{q_0}$. In this way, we obtain a partition ${\mathcal D}$ of the pointset of $\PG((n+1)t-1,q_0)$ by $(t-1)$-dimensional subspaces. In general, a partition of the point set of a projective space by subspaces of a given dimension $d$ is called a {\it spread}, or a {\it $d$-spread} if we want to specify the dimension. The spread obtained by field reduction is called a {\it Desarguesian spread}. Note that the Desarguesian spread satisfies the property that each subspace spanned by spread elements is partitioned by spread elements. 

Let $\mathcal{D}$ be the Desarguesian $(t-1)$-spread of $\PG((n+1)t-1,q_0)$. If $U$ is a subset of $\PG((n+1)t-1,q_0)$, then we define
$\mathcal{B}(U):=\lbrace R \in \mathcal{D}||U\cap R \neq \emptyset \rbrace$, 
and we identify the elements of $\mathcal{B}(U)$ with the corresponding points of $\PG(n,q_0^t)$.
If $U$ is subspace of $\PG((n+1)t-1,q_0)$, then we call ${\mathcal B}(U)$ a {\it linear set} or an {\it ${\mathbb F}_{q_0}$-linear set} if we want to specify the underlying field. %
Note that through every point in $\mathcal{B}(U)$, there is a subspace $U'$ 
 such that $\mathcal{B}(U')=\mathcal{B}(U)$ since the elementwise stabiliser of the Desarguesian spread $\mathcal{D}$ acts transitively on the points of a spread element of $\mathcal{D}$. %
 If $U$ intersects the elements of $\mathcal D$ in at most a point, i.e. $|B(U)|$ is maximal, then we say that $U$ is {\it scattered} with respect to $\mathcal D$; in this case ${\mathcal B}(U)$ is called a {\it scattered linear set}.
We denote the element of $\mathcal{D}$ corresponding to a point $P$ of $\PG(n,q_0^t)$ by $\mathcal{S}(P)$. If $U$ is a subset of $\PG(n,q)$, then we define
$\mathcal{S}(U):=\lbrace \mathcal{S}(P)\vert \vert P \in U \rbrace$.
Analogously to the correspondence between the points of $\PG(n,q_0^t)$, and the elements $\mathcal D$, we obtain the correspondence between the lines of $\PG(n,q)$ and the $(2t-1)$-dimensional subspaces of $\PG((n+1)t-1,q_0)$ spanned by two elements of $\mathcal D$, and in general, we obtain the correspondence between the $(n-k)$-spaces of $\PG(n,q)$ and the $((n-k+1)t-1)$-dimensional subspaces of $\PG((n+1)t-1,q_0)$ spanned by $n-k+1$ elements of $\mathcal D$. With this in mind, it is clear that any $tk$-dimensional subspace $U$ of $\PG(t(n+1)-1,q_0)$ defines a $k$-blocking set $\mathcal{B}(U)$  in $\PG(n,q)$. A ($k$-)blocking set constructed in this way is called a {\it linear ($k$-)blocking set}, or an {\it $\mathbb{F}_{q_0}$-linear ($k$-)blocking set} if we want to specify the underlying field.

By far the most challenging problem concerning blocking sets is the so-called {\it linearity conjecture}. Since 1998 it has been conjectured by many mathematicians working in the field. The conjecture was explicitly stated in the literature by Sziklai in \cite{sziklai}.
\begin{itemize}
\item[{(LC)}]
{\it All small minimal $k$-blocking sets in $\PG(n,q)$ are linear.}
\end{itemize}
Various instances of the conjecture have been proved; for an overview we refer to \cite{sziklai}. In this paper we prove the linearity conjecture for small minimal $k$-blocking sets in $\PG(n,p^3)$, $p\geq 7$, as a corollary of the following main theorem:
\begin{theorem} A small minimal $k$-blocking set in $\PG(n,q^3)$, $q=p^h$, $p$ prime, $h\geq 1$, $p\geq 7$, intersecting every $(n-k)$-space in $1\pmod{q}$ points is linear.
\end{theorem}

%%%%%%%%%%%%%%%%%%%%%%%%%%%%%%%%%%%%%%%%%%%%%%%%%%%

\subsection{Known characterisation results}
In this section we mention a few results, that we will rely on in the sequel of this paper.
% Some of these results prove instances of the linearity conjecture.
First of all, observe that a subspace intersects a linear set of $\PG(n,p^h)$ in  $1 \pmod{p}$ or zero points. The following result of Sz\H{o}nyi and Weiner shows that this property holds for all small minimal blocking sets.
\begin{result} \label{Sz}\cite[Theorem 2.7]{sz} 
If $B$ is a small minimal $k$-blocking set of $\PG(n,q)$, $p>2$, then every subspace intersects $B$ in $1 \pmod{p}$ or zero points.
\end{result}
Result \ref{Sz} answers the linearity conjecture in the affirmative for $\PG(n,p)$.
For $\PG(n,p^2)$, the linearity conjecture was proved by Weiner (see \cite{weiner2}).
For $1$-blocking sets in $\PG(n,q^3)$, we have the following theorem 
of Polverino ($n=2$) and Storme and Weiner ($n\geq 3$).
\begin{result} \label{opmer}\cite{pol2}\label{St}\cite{Storme-Weiner} A minimal $1$-blocking set in $\PG(n,q^3)$, $q=p^h$, $h\geq 1$, $p$ prime, $p\geq 7$, $n\geq 2$, of size at most $q^3+q^2+q+1$, is linear.
%\begin{itemize}
%\item[(a)] a line;
%\item[(b)] a Baer subplane when $q$ is a square;
%\item[(c)] a minimal planar blocking set of cardinality $q^3+q^2+1$;
%\item[(d)] a minimal planar blocking set of cardinality $q^3+q^2+q+1$;
%\item[(e)] a subgeometry $\PG(3,q)$ in a $3$-dimensional subspace of $\PG(n,q^3)$.
%\end{itemize}
\end{result}

In Theorem \ref{basis} we show that this implies the linearity conjecture for small minimal 1-blocking sets $\PG(n,q^3)$, $p\geq7$, that intersect every hyperplane in $1\pmod{q}$ points.

The following Result by Sz\H{o}nyi and Weiner gives a sufficient condition for a blocking set to be minimal.
\begin{result}\label{obs}\cite[Lemma 3.1]{sz} Let $B$ be a $k$-blocking set of $\PG(n,q)$, and suppose that $\vert B \vert \leq 2q^{k}$. If each $(n-k)$-dimensional subspace of $\PG(n,q)$ intersects $B$ in $1 \pmod{p}$ points, then $B$ is minimal. 
\end{result}
%%%%%%%%%%%%%%%%%%%%%%%%%%%%%%%%%%%%%%%%%%%%%%%%%%%

\subsection{The intersection of a subline and an $\F$-linear set}
The possibilities for an $\F$-linear set of $\PG(1,q^3)$, other than the empty set, a point, and the set $\PG(1,q^3)$ itself are the following:
a subline $\PG(1,q)$ of $\PG(1,q^3)$, corresponding to the a line of $\PG(5,q)$ not contained in an element of $\mathcal D$;
a set of $q^2+1$ points of $\PG(1,q^3)$, corresponding to a plane of $\PG(5,q)$ that intersects an element of $\mathcal D$ in a line;
a set of $q^2+q+1$ points of $\PG(1,q^3)$, corresponding to a plane of $\PG(5,q)$ that is scattered w.r.t. $\mathcal D$.

The following results describe the possibilities for the intersection of a subline  with an $\F$-linear set in $\PG(1,q^3)$, and will play an important role in this paper.
\begin{result}\cite{LSV3}\label{gevolg} A subline $\cong \PG(1,q)$ intersects an $\F$-linear set  of $\PG(1,q^3)$ in $0,1,2,3,$ or $q+1$ points.
\end{result}
\begin{result}\cite[Lemma 4.4, 4.5, 4.6]{Leo}\label{ba} Let $q$ be a square. A subline $\PG(1,q)$ and a Baer subline $\PG(1,q\sqrt{q})$ of $\PG(1,q^3)$ share at most a subline $\PG(1,\sqrt{q})$. A Baer subline $\PG(1,q\sqrt{q})$ and an $\F$-linear set of $q^2+1$ or $q^2+q+1$ points in $\PG(1,q^3)$ share at most $q+\sqrt{q}+1$ points. 
\end{result}

%%%%%%%%%%%%%%%%%%%%%%%%%%%%%%%%%%%%%%%%%%%%%%%%%%%

\section{Some bounds and the case $k=1$}

The Gaussian coefficient $\left[\begin{array}{c}n\\k \end{array}\right]_q$ denotes the number of $(k-1)$-subspaces in $\PG(n-1,q)$, i.e.,
$$\left[\begin{array}{c}n\\k \end{array}\right]_q=\frac{(q^n-1)(q^{n-1}-1)\cdots(q^{n-k+1}-1)}{(q^k-1)(q^{k-1}-1)\cdots(q-1)}.$$

\begin{lemma}\label{grootte}
If $B$ is a subset of $\PG(n,q^3)$, $q\geq 7$, intersecting every $(n-k)$-space, $k\geq 1$, in $1 \pmod{q}$ points, and $\pi$ is an $(n-k+s)$-space, $s\leq k$,  then either
$$\vert B\cap \pi \vert< q^{3s}+q^{3s-1}+q^{3s-2}+3q^{3s-3}$$
or
$$\vert B\cap \pi \vert> q^{3s+1}-q^{3s-1}-q^{3s-2}-3q^{3s-3}.$$
\end{lemma}
\begin{proof}
Let $\pi$ be an $(n-k+s)$-space of $\PG(n,q^3)$, and put $B_\pi:=B\cap \pi$.
Let $x_i$ denote the number of $(n-k)$-spaces of $\pi$ intersecting $B_\pi$ in $i$ points. 
Counting the number of $(n-k)$-spaces, the number of incident pairs $(P, \sigma)$ with $
P\in B_\pi, P\in \sigma,\sigma$ an $(n-k)$-space, and the number of triples $(P_1,P_2,\sigma)
$, with $P_1,P_2\in B_\pi$, $P_1\neq P_2$, $P_1,P_2\in \sigma$, $\sigma$ an $(n-k)$-space 
yields:
 \begin{eqnarray}
\sum_i x_i&=&\left[\begin{array}{c}n-k+s+1\\n-k+1 \end{array}\right]_{q^3},\\
\sum_i ix_i&=&\vert B_\pi \vert \left[\begin{array}{c}n-k+s\\n-k \end{array}\right]_{q^3
},\\
\sum i(i-1)x_i&=&\vert B_\pi\vert (\vert B_\pi\vert-1)\left[\begin{array}{c}n-k+s-1\\n-k-1 
\end{array}\right]_{q^3}.
\end{eqnarray}
Since we assume that every $(n-k)$-space intersects $B$ in $1\pmod{q}$ points, it follows that 
every $(n-k)$-space of $\pi$ intersect $B_\pi$ in $1\pmod{q}$ points, and hence
$\sum_i (i-1)(i-1-q)x_i\geq 0$.
Using Equations (1), (2), and (3), this yields that
$$\vert B_\pi \vert (\vert B_\pi\vert -1)(q^{3n-3k}-1)(q^{3n-3k+3}-1)-(q+1)\vert B_\pi \vert (q^{3n-3k+3s}-1)(q^{3n-3k+3}-1)$$
$$+(q+1)(q^{3n-3k+3s+3}-1)(q^{3n-3k+3s}-1)\geq 0.$$ Putting $\vert B_\pi\vert=q^{3s}+q^{3s-1}+q^{3s-2}+3q^{3s-3}$ or $\vert B_\pi\vert=q^{3s+1}-q^{3s-1}-q^{3s-2}-3q^{3s-3}$ in this inequality, with $q\geq7$, gives a contradiction. Hence the statement follows.
\end{proof}

\begin{theorem} \label{basis} A small minimal $1$-blocking set in $\PG(n,q^3)$, $p\geq 7$, intersecting every hyperplane in $1 \pmod{q}$ points, is linear.
\end{theorem} 
\begin{proof}
Lemma \ref{grootte} implies that a small minimal $1$-blocking set $B$ in $\PG(n,q^3)$, intersecting every hyperplane in $1 \pmod{q}$ points, has at most $q^3+q^2+q+3$ points. Since every hyperplane intersects $B$ in $1 \pmod {q}$ points, it is easy to see that $\vert B \vert \equiv 1\pmod{q}$. This implies that $\vert B \vert \leq q^3+q^2+q+1$. Result \ref{opmer} shows that $B$ is linear.
\end{proof}
\begin{corollary}A small minimal $1$-blocking set in $\PG(n,p^3)$, $p$ prime, $p\geq 7$, is $\mathbb{F}_p$-linear.
\end{corollary}
\begin{proof} This follows from Result \ref{Sz} and Theorem \ref{basis}.
\end{proof}

%%%%%%%%%%%%%%%%%%%%%%%%%%%%%%%%%%%%%%%%%%%%%%%%%%%
For the remaining of this section, we use the following assumption:
\begin{itemize}
\item[(B)]{$B$ is small minimal $k$-blocking set in $\PG(n,q^3)$, $p\geq 7$, intersecting every $(n-k)$-space in $1 \pmod{q}$ points.}
\end{itemize}

For convenience let us introduce the following terminology.
A $full$ line of $B$ is a line which is contained in $B$. An $(n-k+s)$-space $S$, $s<k$, is called {\em large} if $S$ contains more than $q^{3s+1}-q^{3s-1}-q^{3s-2}-3q^{3s-3}$ points of $B$, and $S$ is called {\em small} if it contains less than $q^{3s}+q^{3s-1}+q^{3s-2}+3q^{3s-3}$ points of $B$.

\begin{lemma}\label{extra} Let $L$ be a line such that $1<\vert B\cap L \vert<q^3+1.$

(1) For all $i\in \{1,\ldots,n-k\}$ there exists an $i$-space $\pi_i$ on $L$ such that $B\cap \pi_i=B\cap L$.

(2) Let $N$ be a line, skew to $L$. For all $j\in \lbrace 1,\ldots k-2\rbrace$, there exists a small $(n-k+j)$-space $\pi_j$ on $L$, skew to $N$.
\end{lemma}
\begin{proof} 
(1) It follows from Result \ref{Sz} that every subspace on $L$ intersects $B\setminus L$ in zero or at least $p$ points. We proceed by induction on the dimension $i$. The statement obviously holds for $i=1$. Suppose there exists an $i$-space $\pi_i$ on $L$ such that $\pi_i\cap B$=$L\cap B$, with $i\leq n-k-1$. If there is no $(i+1)$-space intersecting $B$ only on $L$, then the number of points of $B$ is at least
$$
|B\cap L|+p(q^{3(n-i)-3}+q^{3(n-i)-6}+\ldots+q^3+1),
$$
but by Lemma \ref{grootte} $\vert B \vert \leq q^{3k}+q^{3k-1}+q^{3k-2}+3q^{3k-3}$. If $i<n-k-1$ this is a contradiction. If $i=n-k-1$ then in the above count we may replace the factor $p$ by a factor $q$, using the hypothesis (B), and hence also in this case we get a contradiction.
We may conclude that there exists an $i$-space $\pi_i$ on $L$ such that $B\cap L=B\cap \pi_i$, $\forall i\in \{1,\ldots,n-k\}$.

(2) Part (1) shows that there is an $(n-k-1)$-space $\pi_{n-k-1}$ on $L$, skew to $N$, such that $B\cap L=B\cap \pi_{n-k-1}$. 
If an $(n-k)$-space through $\pi_{n-k-1}$ contains an extra element of $B$, it contains at least $q^2$ extra elements of $B$, since a line containing 2 points of $B$ contains at least $q+1$ points of $B$. This implies that there is an $(n-k)$-space $\pi_{n-k}$ through $\pi_{n-k-1}$ with no extra points of $B$, and skew to $N$. 

We proceed by induction on the dimension $i$. Lemma \ref{hypervlakken}(1) shows that there are at least $(q^{3k}-1)/(q^3-1)-q^{3k-5}-5q^{3k-6}+1>q^3+1$ small $(n-k+1)$-spaces through $\pi_{n-k}$ which proves the statement for $i=1$.

Suppose that there exists an $(n-k+t)$-space $\pi_{n-k+t}$ on $L$, skew to $N$, such that $B\cap \pi_{n-k+t}$ is a small minimal $t$-blocking set of $\pi_{n-k+t}$. An $(n-k+t+1)$-space through $\pi_{n-k+t}$ contains at most $(q^{3t+4}-1)(q-1)$ or more than $q^{3t+4}-q^{3t+2}-q^{3t+1}-3q^{3t}$ points of $B$ (see Lemmas \ref{grootte} and \ref{kleine}).

Suppose all $(q^{3k-3t}-1)(q^3-1)-q^3-1$ $(n-k+t)$-spaces through $\pi_{n-k+t-1}$, skew to $N$, contain more than $q^{3t+4}-q^{3t+2}-q^{3t+1}-3q^{3t}$ points of $B$. Then the number of points in $B$ is larger than $q^{3k}+q^{3k-1}+q^{3k-2}+3q^{3k-3}$ if $t\leq k-3$, a contradiction. 

We may conclude that there exists an $(n-k+j)$-space $\pi_j$ on $L$ such that $B\cap \pi_j$ is a small minimal $i$-blocking set, skew to $N$, $\forall j\in \{1,\ldots,k-2\}$.
\end{proof}

\begin{theorem}\label{inters}
A line $L$ intersects $B$ in a linear set.
%$0,1,q+1,q\sqrt{q}+1$ (if $q$ is a square), $q^2+1,q^2+q+1,$ or $q^3+1$ points.
\end{theorem}
\begin{proof} 
Note that it is enough to show that $L$ is contained in a subspace of $\PG(n,q^3)$ intersecting $B$ in a linear set.
If $k=1$, then $B$ is linear by Theorem \ref{basis}, and the statement follows.
Let $k>1$, let $L$ be a line, not contained in $B$, intersecting $B$ in at least two points. 
It follows from Lemma \ref{extra} that there exists an $(n-k)$-space $\pi_L$ such that $B\cap L=B\cap \pi_L$.
If each of the $(q^{3k}-1)/(q^3-1)$ $(n-k+1)$-spaces through $\pi_L$ is large, then the number of points in $B$ is at least
$$\frac{q^{3k}-1}{q^3-1}(q^4-q^2-q-3-q^3)+q^3>q^{3k}+q^{3k-1}+q^{3k-2}+3q^{3k-3},$$
a contradiction.
Hence, there is a small $(n-k+1)$-space $\pi$ through $L$, so $B\cap \pi$ is a small $1$-blocking set which is linear by Theorem \ref{basis}. This concludes the proof.
\end{proof}

\begin{lemma}\label{hypervlakken} Let $\pi$ be an $(n-k)$-space of $\PG(n,q^3)$, $k>1$.
\begin{enumerate}
\item[(1)] If $B\cap \pi$ is a point, then there are at most $q^{3k-5}+4q^{3k-6}-1$ large $(n-k+1)$-spaces through $\pi$.
\item[(2)] If $\pi$  intersects $B$ in $(q\sqrt{q}+1)$, $q^2+1$ or $q^2+q+1$ collinear points, then there are at most $q^{3k-5}+5q^{3k-6}-1$ large $(n-k+1)$-spaces through $\pi$.
\item[(3)] If $\pi$ intersects $B$ in $q+1$ collinear points, then there are at most $3q^{3k-6}-q^{3k-7}-1$ large $(n-k+1)$-spaces through $\pi$.

\end{enumerate}
\end{lemma}

\begin{proof}
%From Lemma \ref{grootte} (with $s=1$) we know that an $(n-k+1)$-space through $\pi$ intersects $B$ in at most $q^3+q^2+q+3$ (and at least $q^3+1$ points) or in more than $q^4-q^2-q-3$ points.

Suppose there are $y$ large $(n-k+1)$-spaces through $\pi$. Then the number of points in $B$ is at least
$$ y(q^4-q^2-q-3-\vert B \cap \pi\vert)+((q^{3k}-1)/(q^3-1)-y)x+\vert B \cap \pi\vert, \ (\ast)$$
where $x$ depends on the intersection $B \cap \pi$.

(1) In this case, $x=q^3$ and $\vert B\cap \pi\vert=1$. If $y=q^{3k-5}+4q^{3k-6}$, then $(\ast)$ is larger than $q^{3k}+q^{3k-1}+q^{3k-2}+3q^{3k-3}$, a contradiction.

(2) In this case $x=q^3$ and $\vert B\cap \pi\vert \leq q^2+q+1$. If $y=q^{3k-5}+5q^{3k-6}$, then $(\ast)$ is larger than $q^{3k}+q^{3k-1}+q^{3k-2}+3q^{3k-3}$, a contradiction.

(3) By Result \ref{St} we know that an $(n-k+1)$-space $\pi'$ through $\pi$ intersects $B$ in at least $q^3+q^2+1$ points, since a $(q+1)$-secant in $\pi'$ implies that the intersection of $\pi'$ with $B$ is non-trivial and not a Baer subplane, hence $x=q^3+q^2-q$, and $\vert B \cap \pi\vert=q+1$. If $3q^{3k-6}-q^{3k-7}$, then $(\ast)$ is larger than $q^{3k}+q^{3k-1}+q^{3k-2}+3q^{3k-3}$, a contradiction.
\end{proof}

\section{The proof of Theorem 1}
In the proof of the main theorem, we distinguish two cases. In both cases we need the following two lemmas.

We continue with the following assumption
\begin{itemize}
\item[(B)]{$B$ is small minimal $k$-blocking set in $\PG(n,q^3)$, $p\geq 7$, intersecting every $(n-k)$-space in $1 \pmod{q}$ points;}
\end{itemize}

and we consider the following properties:
\begin{itemize}
\item[$(H_1)$]{$\forall s<k$: every small minimal $s$-blocking set, intersecting every $(n-s)$-space in $1 \pmod{q}$ points, not containing a $(q\sqrt{q}+1)$-secant, is $\mathbb{F}_q$-linear;}
\item[$(H_2)$]{$\forall s<k$: every small minimal $s$-blocking set, intersecting every $(n-s)$-space in $1 \pmod{q}$ points, containing a $(q\sqrt{q}+1)$-secant, is $\mathbb{F}_{q\sqrt{q}}$-linear.}
\end{itemize}

\begin{lemma} \label{kleine}If $(H_1)$ or $(H_2)$, and $S$ is a small $(n-k+s)$-space, $0< s< k$, then $B \cap S$ is a small minimal linear $s$-blocking set in $S$, and hence $|B\cap S|\leq (q^{3s+1}-1)/(q-1)$.
\end{lemma}
\begin{proof}
Clearly $B\cap S$ is an $s$-blocking set in $S$. Result \ref{Sz} implies that $B\cap S$ intersects every $(n-k+s-s)$-space of $S$ in $1\pmod{p}$ points, and it follows from Result \ref{obs} that $B\cap S$ is minimal. Now apply $(H_1)$ or $(H_2)$.
\end{proof}

\begin{lemma}\label{p-3} Suppose $(H_1)$ or $(H_2)$. Let $k>2$ and let $\pi_{n-2}$ be an $(n-2)$-space such that $B\cap \pi_{n-2}$ is a non-trivial small linear $(k-2)$-blocking set, then there are at least $q^3-q+6$ small hyperplanes through $\pi_{n-2}$.
\end{lemma}
\begin{proof} Applying Lemma \ref{kleine} with $s=k-2$, it follows that $B\cap \pi_{n-2}$ contains at most $(q^{3k-5}-1)/(q-1)$ points. On the other hand, from Lemmas \ref{grootte} and \ref{kleine} with $s=k-1$, we know that a hyperplane intersects $B$ in at most $(q^{3k-2}-1)/(q-1)$ points or in more than $q^{3k-2}-q^{3k-4}-q^{3k-5}-3q^{3k-6}$ points. 
In the first case, a hyperplane $H$ intersects $B$ in at least $q^{3k-3}+1+(q^{3k-3}+q)/(q+1)$ points, using a result of Sz\H{o}nyi and Weiner \cite[Corollary 3.7]{sz} for the $(k-1)$-blocking set $H\cap B$. 
If there are at least $q-4$ large hyperplanes, then the number of points in $B$ is at least
$$(q-4)(q^{3k-2}-q^{3k-4}-q^{3k-5}-3q^{3k-6}-\frac{q^{3k-5}-1}{q-1})+$$
$$(q^3-q+5)(q^{3k-3}+1+\frac{q^{3k-3}+q}{q+1}-\frac{q^{3k-5}-1}{q-1})+\frac{q^{3k-5}-1}{q-1},$$ which is larger than $q^{3k}+q^{3k-1}+q^{3k-2}+3q^{3k-3}$ if $q\geq7$, a contradiction. Hence, there are at most $q-5$ large hyperplanes through $\pi_{n-2}$.
\end{proof}

\subsection{Case 1: there are no $q\sqrt{q}+1$-secants}
In this subsection, we will use induction on $k$ to prove that small minimal $k$-blocking sets in $\PG(n,q^3)$, intersecting every $(n-k)$-space in $1\pmod{q}$ points and not containing a $(q\sqrt{q}+1)$-secant, are $\F$-linear. The induction basis is Theorem \ref{basis}. We continue with assumptions $(H_1)$ and
\begin{itemize}
\item[$(B_1)$]{$B$ is small minimal $k$-blocking set in $\PG(n,q^3)$, $p\geq 7$, intersecting every $(n-k)$-space in $1 \pmod{q}$ points, not containing a $(q\sqrt{q}+1)$-secant.}
\end{itemize}
%%%%%%%%%%%%%%%%%%%%%%%%%%%%%%%%%%%%%%%%%%%%%%%%%%%

%\subsection{The intersection of a small minimal blocking set with a line}
%In this subsection, assuming (B) and (H), we will show that the intersection of $B$ with a line is a linear subset of this line.

%%%%%%%%%%%%%%%%%%%%%%%%%%%%%%%%%%%%%%%%%%%%%%%%%%%

%\subsection{Bounds on the number of small spaces}

%%%%%%%%%%%%%%
%The proof of the main theorem uses the following idea. First, we will prove that there are many small $(n-k+1)$-spaces in which there lies a $q+1$-secant (Lemma \ref{situatie}). The next goal is to show that the span of two of such $(n-k+1)$-spaces with a linear intersection with the minimal blocking set $B$, has a linear intersection with $B$ as well.

%Finally, we prove the main theorem by putting small $(n-k+1)$-spaces together, such that the intersection of the obtained space with the minimal blocking set $B$ is linear, until the chosen $(n-k+1)$-spaces span the whole space, hence showing that the minimal blocking set $B$ contains a linear blocking set.

\begin{lemma}\label{situatie} If $B$ is non-trivial, there exist a point $P\in B$, a tangent $(n-k)$-space $\pi$ at the point $P$  and  small $(n-k+1)$-spaces $H_i$, through $\pi$, such that there is a $(q+1)$-secant through $P$ in $H_i$, $i=1,\ldots,q^{3k-3}-2q^{3k-4}$.
\end{lemma}
\begin{proof} Since $B$ is non-trivial, there is at least one line $N$ with $1<\vert N \cap B\vert <q^3+1$. Lemma \ref{extra} shows that there is an $(n-k)$-space $\pi_N$ through $N$ such that $B\cap N=B\cap \pi_N$.  It follows from Theorem \ref{inters} and Lemma \ref{hypervlakken} that there is at least one $(n-k+1)$-space $H$ through $\pi_N$ such that $H \cap B $ is a small minimal linear $1$-blocking set of $H$. In this non-trivial small minimal linear $1$-blocking set, there are $(q+1)$-secants (see Result \ref{St}).  Let $M$ be one of those $(q+1)$-secants of $B$. Again using Lemma \ref{extra}, we find an $(n-k)$-space $\pi_M$ through $M$ such that $B\cap M=B\cap \pi_M$.

Lemma \ref{hypervlakken}(3) shows that through $\pi_M$, there are at least $\frac{q^{3k}-1}{q^3-1}-3q^{3k-6}+q^{3k-7}+1$ small $(n-k+1)$-spaces. Let $P$ be a point of $M$. Since in each of these intersections, $P$ lies on at least $q^2-1$ other $(q+1)$-secants, a point $P$ of $M$ lies in total on at least $(q^2-1)(\frac{q^{3k}-1}{q^3-1}-3q^{3k-6}+q^{3k-7}+1)$ other $(q+1)$-secants. Since each of the $\frac{q^{3k}-1}{q^3-1}-3q^{3k-6}+q^{3k-7}+1$ small $(n-k+1)$-spaces contains at least $q^3+q^2-q$ points of $B$ not on $M$, and $\vert B \vert <q^{3k}+q^{3k-1}+q^{3k-2}+3q^{3k-3}$ (see Lemma \ref{grootte}), there are less than $2q^{3k-2}+6q^{3k-3}$ points of $B$ left in the large $(n-k+1)$-spaces. Hence, $P$ lies on less than $2q^{3k-5}+6q^{3k-6}$ full lines.

Since $B$ is minimal, $P$ lies on a tangent $(n-k)$-space $\pi$. There are at most $q^{3k-5}+4q^{3k-6}-1$ large $(n-k+1)$-spaces through $\pi$ (Lemma \ref{hypervlakken}(1)). Moreover, since at least $\frac{q^{3k}-1}{q^3-1}-(q^{3k-5}+4q^{3k-6}-1)-(2q^{3k-5}+6q^{3k-6})$ $(n-k+1)$-spaces through $\pi$ contain at least $q^3+q^2$ points of $B$, and at most $2q^{3k-5}+6q^{3k-6}$ of the small $(n-k+1)$-spaces through $\pi$ contain exactly $q^3+1$ points of $B$, there are at most $2q^{3k-2}+23q^{3k-3}$ points of $B$ left. Hence, $P$ lies on at most $2q^{3k-3}+23q^{3k-4}$ $(q+1)$-secants of the large $(n-k+1)$-spaces through $\pi$. This implies that there are at least 
$(q^2-1)(\frac{q^{3k}-1}{q^3-1}-3q^{3k-6}+q^{3k-7}+1)-(2q^{3k-3}+23q^{3k-4})$
 $(q+1)$-secants through $P$ left in small $(n-k+1)$-spaces through $\pi$. Since in a small $(n-k+1)$-space through $\pi$, there can lie at most $q^2+q+1$ $(q+1)$-secants through $P$, this implies that there are at least $q^{3k-3}-2q^{3k-4}$ $(n-k+1)$-spaces $H_i$ through $\pi$ such that $P$ lies on a $(q+1)$-secant in $H_i$.
\end{proof}

\begin{lemma} \label{ess} Let $\pi$ be an $(n-k)$-dimensional tangent space of $B$ at the point $P$. Let $H_1$ and $H_2$ be two $(n-k+1)$-spaces through $\pi$ for which $B\cap H_i=\mathcal{B}(\pi_i)$, for some $3$-space $\pi_i$ through $x\in \mathcal{S}(P)$, $\mathcal{B}(x)\cap \pi_i=\lbrace x \rbrace$ $(i=1,2)$ and $\mathcal{B}(\pi_i)$ not contained in a line of $\PG(n,q^3)$.
Then $\mathcal{B}(\langle \pi_1,\pi_2\rangle)\subseteq B$.

\end{lemma}
\begin{proof}
Since $\langle \mathcal{B}(\pi_i)\rangle$ is not contained in a line of $\PG(n,q^3)$, there is at most one element $Q$ of $\mathcal{B}(\pi_i)$ such that $\langle \mathcal{S}(P),Q\rangle$ intersects $\pi_i$ in a plane. If there is such a plane, then we denote its pointset by $\mu_i$, otherwise we put $\mu_i=\emptyset$.

Let $M$ be a line through $x$ in $\pi_1\setminus \mu_1$, let $s\neq x$ be a point of $\pi_2\setminus \mu_2$, and note that $\mathcal{B}(s)\cap \pi_2=\lbrace s \rbrace$.

We claim that there is a line $T$ through $s$ in $\pi_2$ and an $(n-2)$-space $\pi_M$ through $\langle\mathcal{B}(M)\rangle$ such that there are at least $4$ points $t_i\in T, t_i\notin \mu_2$, such that $\langle \pi_M, \mathcal{B}(t_i)\rangle$ is small and hence has a linear intersection with $B$, with $B\cap \pi_M=M$ if $k=2$ and $B\cap \pi_M$ is a small minimal $(k-2)$-blocking set if $k>2$.

If $k=2$, the existence of $\pi_M$ follows from Lemma \ref{extra}(1), and we know from Lemma \ref{hypervlakken}(1) that there are at most $q+3$ large hyperplanes through $\pi_M$. Denote the set of points of $\mathcal{B}(\pi_2)$, contained in one of those hyperplanes by $F$. Hence, if $Q$ is a point of $\mathcal{B}(\pi_2)\setminus F$, $\langle Q,\pi_M \rangle$ is a small hyperplane.

Let $T_1$ be a line through $s$ in $\pi_2\setminus \mu_2$ and not through $x$, and suppose that $\mathcal{B}(T_1)$ contains at least $q-3$ points of $F$.

Let $T_2$ be a line in $\pi_2\setminus \mu_2$, through $s$, not in $\langle x, T_1\rangle$, not through $x$. There are at most $q+3-(q-3)$ reguli through $x$ of $\mathcal{S}(F)$, not in $\langle x,T_1\rangle$, and if $\mu\neq \emptyset$ one element of $\mathcal{B}(\mu_2)$ is contained $\mathcal{B}(T_2)$. Since it is possible that $\mathcal{B}(s)$ is an element of $F$, this gives in total at most $8$ points of $\mathcal{B}(T_2)$ that are contained in $F$. This implies, if $q> 11$, that at least 5 of the hyperplanes $\lbrace \langle \pi_M,\mathcal{B}(t)\rangle||t \in T_2 \rbrace$ are small.

If $q=11$, it is possible that $\mathcal{B}(T_2)$ contains at least $8$ points of $F$. If $T_3$ is a line in $\pi_2\setminus \mu_2$, through $s$, $\langle x,T_1\rangle$, $\langle x,T_2\rangle$ and not through $x$, then there are at least $5$ points $t$ of $T_3$ such that $\langle \pi_M,\mathcal{B}(t)\rangle$ is a small hyperplane.

If $q=7$ and if $\mathcal{B}(s)\in \mathcal{B}(F)$, it is possible that $\mathcal{B}(T_2)$,$\mathcal{B}(T_3)$, and $\mathcal{B}(T_4)$, with $T_i$ a line through $s$ in $\pi_2\setminus \mu_2$, not in $\langle x,T_j\rangle$, $j<i$, not through $x$, contain 4 points of $F$. A fifth line $T_5$ through $s$ in $\pi_2\setminus\mu_2$, not in $\langle x,T_j\rangle$, $j<i$, not through $x$, contains at least 5 points $t$ such that $\langle \pi_M,\mathcal{B}(t)\rangle$ is a small hyperplane.

If $k>2$, let $T$ be a line through $s$ in $\pi_2\setminus \mu_2$, not through $x$. It follows from Lemma \ref{extra}(2) that there is an $(n-2)$-space $\pi_M$ through $\langle\mathcal{B}(M)\rangle$ such that $B\cap \pi_M$ is a small minimal $(k-2)$-blocking set of $\PG(n,q^3)$, skew to $\mathcal{B}(T)$. Lemma \ref{p-3} shows that at most $q-5$ of the hyperplanes through $\pi_M$ are large. This implies that at least 5 of the hyperplanes $\lbrace \langle\pi_M, \mathcal{B}(t)\rangle||t \in \mathcal{B}(T) \rbrace$ are small. This proves our claim.

Since $B\cap \langle \mathcal{B}(t_i),\pi_M\rangle$ is linear, also the intersection of $\langle \mathcal{B}(t_i),\mathcal{B}(M)\rangle$ with $B$ is linear, i.e., there exist subspaces $\tau_i$, $\tau_i\cap \mathcal{S}(P)=\lbrace x \rbrace$, such that $\mathcal{B}(\tau_i)=\langle \mathcal{B}(t_i),\mathcal{B}(M)\rangle \cap B$. Since $\tau_i \cap \langle \mathcal{B}(M)\rangle$ and $M$ are both transversals through $x$ to the same regulus $\mathcal{B}(M)$, they coincide, hence $M\subseteq \tau_i$. The same holds for $\tau_i\cap \langle \mathcal{B}(t_i),\mathcal{S}(P)\rangle$, implying $t_i\in \tau_i$. We conclude that $\mathcal{B}(\langle M,t_i\rangle)\subseteq \mathcal{B}(\tau_i)\subseteq B$.

We show that $\mathcal{B}(\langle M,T\rangle)\subseteq B$. Let $L'$ be a line of $\langle M,T\rangle$, not intersecting $M$. The line $L'$ intersects the planes $\langle M,t_i\rangle$ in points $p_i$ such that $\mathcal{B}(p_i)\in B$. Since $\mathcal{B}(L')$ is a subline intersecting $B$ in at least 4 points, Result \ref{gevolg} shows that $\mathcal{B}(L')\subset B$. Since every point of the space $\langle M,T\rangle$ lies on such a line $L'$, $\mathcal{B}(\langle M,T\rangle)\subseteq B$. 

Hence, $\mathcal{B}(\langle M,s\rangle)\subseteq B$ for all lines $M$ through $x$, $M$ in $\pi_1\setminus \mu_1$, and all points $s\neq x\in\pi_2\setminus \mu_2$, so $\mathcal{B}(\langle \pi_1,\pi_2\rangle\setminus (\langle\mu_1,\pi_2\rangle\cup \langle\mu_2,\pi_1\rangle))\subseteq B$. Since every point of $\langle\mu_1,\pi_2\rangle\cup \langle\mu_2,\pi_1\rangle$ lies on a line $N$ with $q-1$ points of $\langle \pi_1,\pi_2\rangle\setminus (\langle\mu_1,\pi_2\rangle\cup \langle\mu_2,\pi_1\rangle)$, Result \ref{gevolg} shows that $\mathcal{B}(N)\subset B$. We conclude that $\mathcal{B}(\langle\pi_1,\pi_2\rangle)\subseteq B$.
\end{proof}
%%%%%%%%%%%%%

\begin{theorem}\label{H1} The set $B$ is $\mathbb{F}_q$-linear.
\end{theorem}
\begin{proof} If $B$ is a $k$-space, then $B$ is $\F$-linear. If $B$ is non-trivial small minimal $k$-blocking set, Lemma \ref{situatie} shows that there exists a point $P$ of $B$, a tangent $(n-k)$-space $\pi$ at the point $P$ and at least $q^{3k-3}-2q^{3k-4}$ $(n-k+1)$-spaces $H_i$  through $\pi$ for which $B\cap H_i$ is small and linear, where $P$ lies on at least one $(q+1)$-secant of $B\cap H_i$, $i=1,\ldots,s$, $s\geq q^{3k-3}-2q^{3k-4}$. Let $B\cap H_i=\mathcal{B}(\pi_i), i=1,\ldots, s$, with $\pi_i$ a $3$-dimensional space.

Lemma \ref{ess} shows that $\mathcal{B}(\langle \pi_i,\pi_j \rangle)\subseteq B$, $0\leq i\neq j\leq s$. 

If $k=2$, the set $\mathcal{B}(\langle \pi_1,\pi_2 \rangle)$ corresponds to a linear $2$-blocking set $B'$ in $\PG(n,q^3)$. Since $B$ is minimal, $B=B'$, and the Theorem is proven.

Let $k>2$. Denote the $(n-k+1)$-spaces through $\pi$, different from $H_i$, by $K_j, j=1,\ldots, z$. It follows from Lemma \ref{situatie} that $z\leq 2q^{3k-4}+(q^{3k-3}-1)/(q^3-1)$.
There are at least $(q^{3k-3}-2q^{3k-4}-1)/q^3$ different $(n-k+2)$-spaces $\langle H_1, H_j\rangle$, $1<j\leq s$. If all $(n-k+2)$-spaces $\langle H_1,H_j\rangle$, contain at least $5q^2-49$ of the spaces $K_i$, then $z\geq
(5q^2-49)(q^{3k-3}-2q^{3k-4}-1)/q^3$, a contradiction if $q\geq 7$. Let $\langle H_1,H_2\rangle$ be an $(n-k+2)$-spaces containing less than $5q^2-49$ spaces $K_i$. 
 
Suppose by induction that for any $1<i<k$, there is an $(n-k+i)$-space $\langle H_1,H_2,\ldots, H_i\rangle$  containing at most $5q^{3i-4}-49q^{3i-6}$ of the spaces $K_i$ such that $\mathcal{B}(\langle \pi_1,\ldots,\pi_i\rangle)\subseteq B$. 

There are at least $\frac{q^{3k-3}-2q^{3k-4}-(q^{3i}-1)/(q^3-1)}{q^{3i}}$ different $(n-k+i+1)$-spaces $\langle H_1,H_2,\ldots, H_i,H\rangle$, $H\not\subseteq \langle H_1,H_2,\ldots, H_i\rangle$. If all of these contain at least $5q^{3i-1}-49q^{3i-3}$ of the spaces $K_i$, then 
\begin{displaymath}
\begin{array}{ll}
z\geq & (5q^{3i-1}-49q^{3i-3}-5q^{3i-4}+49q^{3i-6}) \frac{q^{3k-3}-2q^{3k-4}-(q^{3i}-1)/(q^3-1)}{q^{3i}}\\ 
& +5q^{3i-4}-49q^{3i-6},\\
\end{array}
\end{displaymath}
 a contradiction if $q\geq 7$. Let $\langle H_1,\ldots, H_{i+1}\rangle$ be an $(n-k+i+1)$-space containing less than $5q^{3i-1}-49q^{3i-3}$ spaces $K_i$. We still need to prove that $\mathcal{B}(\langle \pi_1,\ldots,\pi_{i+1}\rangle)\subseteq B$. Since $\mathcal{B}(\langle \pi_{i+1}, \pi \rangle)\subseteq B$, with $\pi$ a 3-space in $ \langle \pi_1,\ldots,\pi_i\rangle$ for which $\mathcal{B}(\pi)$ is not contained in one of the spaces $K_i$, there are at most $5q^{3i-4}-49q^{3i-6}$ 6-dimensional spaces $\langle \pi_{i+1},\mu\rangle$ for which $\mathcal{B}(\langle\pi_{i+1},\mu\rangle)$ is not necessarily contained in $B$, giving rise to at most $(5q^{3i-4}-49q^{3i-6})(q^6+q^5+q^4)$ points $t$ for which $\mathcal{B}(t)$ is not necessarily contained in $B$.  Let $u$ be a point of such a space $\langle \pi_{i+1},\mu \rangle$. Suppose that each of the $(q^{3i+3}-1)/(q-1)$ lines through $u$ in $\langle \pi_1,\ldots, \pi_{i+1}\rangle$ contains at least $q-2$ of the points $t$ for which $\mathcal{B}(t)$ is not in $B$. Then there are at least $(q-3)(q^{3i+3}-1)/(q-1)+1>(5q^{3i-4}-49q^{3i-6})(q^6+q^5+q^4)$ such points $t$, if $q\geq 7$, a contradiction. Hence, there is a line $N$ through $t$ for which for at least 4 points $v\in N$, $\mathcal{B}(v)\in B$. Result \ref{gevolg} yields that $\mathcal{B}(t)\in B$. This implies that $\mathcal{B}(\langle \pi_1,\ldots,\pi_{i+1}\rangle)\subseteq B$.
 
Hence, the space $\langle H_1,H_2,\ldots, H_k\rangle$, which spans the space $\PG(n,q^3)$, is such that $\mathcal{B}(\langle \pi_1,\ldots,\pi_k\rangle)\subseteq B$.  But $\mathcal{B}(\langle \pi_1,\ldots, \pi_k \rangle)$ corresponds to a linear $k$-blocking set $B'$ in $\PG(n,q^3)$. Since $B$ is minimal, $B=B'$.
\end{proof}

\begin{corollary}A small minimal $k$-blocking set in $\PG(n,p^3)$, $p$ prime, $p\geq 7$, is $\mathbb{F}_p$-linear.
\end{corollary}
\begin{proof} This follows from Results \ref{Sz} and Theorem \ref{H1}.
\end{proof}

%%%%%%%%%%%%%%%%%%%%%%%%%%%%%%%%%%%%%%%%%%%%%%%%%%

\subsection{Case 2: there are $(q\sqrt{q}+1)$-secants to $B$}
In this subsection, we will use induction on $k$ to prove that small minimal $k$-blocking sets in $\PG(n,q^3)$, intersecting every $(n-k)$-space in $1\pmod{q}$ points and containing a $q\sqrt{q}+1$-secant, are $\mathbb{F}_{q\sqrt{q}}$-linear. The induction basis is Theorem \ref{basis}. We continue with assumptions $(H_2)$ and
\begin{itemize}
\item[($B_2$)]{$B$ is small minimal $k$-blocking set in $\PG(n,q^3)$ intersecting every $(n-k)$-space in $1 \pmod{q}$ points, containing a $(q\sqrt{q}+1)$-secant.}
\end{itemize}

In this case, $\mathcal{S}$ maps $\PG(n,q^3)$ onto $\PG(2n+1,q\sqrt{q})$ and the Desarguesian spread consists of lines.

\begin{lemma}\label{situatie'} If $B$ is non-trivial, there exist a point $P\in B$, a tangent $(n-k)$-space $\pi$ at $P$ and  small $(n-k+1)$-spaces $H_i$ through $\pi$, such that there is a $(q\sqrt{q}+1)$-secant through $P$ in $H_i$, $i=1,\ldots,q^{3k-3}-q^{3k-4}-2\sqrt{q}q^{3k-5}$.
\end{lemma}
\begin{proof}

There is a $(q\sqrt{q}+1)$-secant $M$. Lemma \ref{extra}(1) shows that there is an $(n-k)$-space $\pi_M$ through $M$ such that $B\cap M=B\cap \pi_M$.

Lemma \ref{hypervlakken}(3) shows that there are at least $\frac{q^{3k}-1}{q^3-1}-q^{3k-5}-5q^{3k-6}+1$ small $(n-k+1)$-spaces through $\pi_M$. Moreover, the intersections of these small $(n-k+1)$-spaces with $B$ are Baer subplanes $\PG(2,q\sqrt{q})$, since there is a $(q\sqrt{q}+1)$-secant $M$. Let $P$ be a point of $M\cap B$.

Since in any of these intersections, $P$ lies on $q\sqrt{q}$ other $(q\sqrt{q}+1)$-secants, a point $P$ of $M\cap B$ lies in total on at least $q\sqrt{q}(\frac{q^{3k}-1}{q^3-1}-q^{3k-5}-5q^{3k-6}+1)$ other $(q\sqrt{q}+1)$-secants. Since any of the $\frac{q^{3k}-1}{q^3-1}-q^{3k-5}-5q^{3k-6}+1$ small $(n-k+1)$-spaces through $\pi_M$ contains $q^3$ points of $B$ not in $\pi_M$, and $\vert B \vert <q^{3k}+q^{3k-1}+q^{3k-2}+3q^{3k-3}$ (see Lemma \ref{grootte}), there are less than $q^{3k-1}+4q^{3k-2}$ points of $B$ left in the other $(n-k+1)$-spaces through $\pi_M$. Hence, $P$ lies on less than $q^{3k-4}+4q^{3k-5}$ full lines.

Since $B$ is minimal, there is a tangent $(n-k)$-space $\pi$ through $P$. There are at most $q^{3k-5}+4q^{3k-6}-1$ large $(n-k+1)$-spaces through $\pi$ (Lemma \ref{hypervlakken}(1)). Moreover, since at least $\frac{q^{3k}-1}{q^3-1}-(q^{3k-5}+4q^{3k-6}-1)-(q^{3k-4}+4q^{3k-5})$ small $(n-k+1)$-spaces through $\pi$ contain $q^3+q\sqrt{q}+1$ points of $B$, and at most $q^{3k-4}+4q^{3k-5}$ of the small $(n-k+1)$-spaces through $\pi$ contain exactly $q^3+1$ points of $B$, there are at most $q^{3k-1}-q^{3k-2}\sqrt{q}+4q^{3k-2}$ points of $B$ left. Hence, $P$ lies on at most $(q^{3k-1}-q^{3k-2}\sqrt{q}+4q^{3k-2})/(q\sqrt{q}+1)$ different $(q\sqrt{q}+1)$-secants of the large $(n-k+1)$-spaces through $\pi$. This implies that there are at least $q\sqrt{q}(\frac{q^{3k}-1}{q^3-1}-q^{3k-5}-5q^{3k-6}+1)-(q^{3k-1}-q^{3k-2}\sqrt{q}+4q^{3k-2})/(q\sqrt{q}+1)$ different $(q\sqrt{q}+1)$-secants left through $P$ in small $(n-k+1)$-spaces through $\pi$. Since in a small $(n-k+1)$-space through $\pi$, there lie $q\sqrt{q}+1$ different $(q\sqrt{q}+1)$-secants through $P$, this implies that there are certainly at least $q^{3k-3}-q^{3k-4}-2\sqrt{q}q^{3k-5}$ small $(n-k+1)$-spaces $H_i$ through $\pi$ such that $P$ lies on a $(q\sqrt{q}+1)$-secant in $H_i$.
\end{proof}

\begin{lemma}\label{ess'} Let $\pi$ be an $(n-k)$-dimensional tangent space of $B$ at the point $P$. Let $H_1$ and $H_2$ be two $(n-k+1)$-spaces through $\pi$ for which $B\cap H_i=\mathcal{B}(\pi_i)$, for some plane $\pi_i$ through $x\in \mathcal{S}(P)$, $\mathcal{B}(x)\cap \pi_i=\lbrace x \rbrace$ $(i=1,2)$ and $\mathcal{B}(\pi_i)$ not contained in a line of $\PG(n,q^3)$.
Then $\mathcal{B}(\langle \pi_1,\pi_2\rangle)\subseteq B$.

\end{lemma}

\begin{proof}
Let $M$ be a line through $x$ in $\pi_1$, let $s\neq x$ be a point of $\pi_2$.

We claim that there is a line $T$ through $s$, not through $x$, in $\pi_2$ and an $(n-2)$-space $\pi_M$ through $\langle\mathcal{B}(M)\rangle$ such that there are at least $q\sqrt{q}-q-2$ points $t_i\in T$, such that $\langle \pi_M, \mathcal{B}(t_i)\rangle$ is small and hence has a linear intersection with $B$, with $B\cap \pi_M=M$ if $k=2$ and $B\cap \pi_M$ is a small minimal $(k-2)$-blocking set if $k>2$.
From Lemma \ref{hypervlakken}(1), we know that there are at most $q+3$ large hyperplanes through $\pi_M$ if $k=2$, and at most $q-5$ if $k>2$ (see Lemma \ref{p-3}). 

Let $T$ be a line through $s$ in $\pi_2$, not through $x$. The existence of $\pi_M$ follows from Lemma \ref{extra}(1) if $k=2$, and Lemma \ref{extra}(2) if $k>2$. Since $\mathcal{B}(T)$ contains $q\sqrt{q}+1$ spread elements, there are at least $q\sqrt{q}-q-2$ points $t_i\in T$ such that $\langle \pi_M, \mathcal{B}(t_i)\rangle$ is small.
This proves our claim.

Since $B\cap \langle \mathcal{B}(t_i),\pi_M\rangle$ is linear, also the intersection of $\langle \mathcal{B}(t_i),\mathcal{B}(M)\rangle$ with $B$ is linear, i.e., there exist subspaces $\tau_i$, $\tau_i\cap \mathcal{S}(P)=\lbrace x \rbrace$, such that $\mathcal{B}(\tau_i)=\langle \mathcal{B}(t_i),\mathcal{B}(M)\rangle \cap B$. Since $\tau_i \cap \langle \mathcal{B}(M)\rangle$ and $M$ are both transversals through $x$ to the same regulus $\mathcal{B}(M)$, they coincide, hence $M\subseteq \tau_i$. The same holds for $\tau_i\cap \langle \mathcal{B}(t_i),\mathcal{S}(P)\rangle$, implying $t_i\in \tau_i$. We conclude that $\mathcal{B}(\langle M,t_i\rangle)\subseteq \mathcal{B}(\tau_i)\subseteq B$.

We show that $\mathcal{B}(\langle M,T\rangle)\subseteq B$. Let $L'$ be a line of $\langle M,T\rangle$, not intersecting $M$. The line $L'$ intersects the planes $\langle M,t_i\rangle$ in points $p_i$ such that $\mathcal{B}(p_i)\subseteq B$. Since $\mathcal{B}(L')$ is a subline intersecting $B$ in at least $q\sqrt{q}-q-2$ points, Result \ref{ba} shows that $\mathcal{B}(L')\subseteq B$. Since every point of the space $\langle M,T\rangle$ lies on such a line $L'$, $\mathcal{B}(\langle M,T\rangle)\subseteq B$. 

Hence, $\mathcal{B}(\langle M,s\rangle)\subseteq B$ for all lines $M$ through $x$ in $\pi_2$, and all points $s\neq x\in\pi_2$. We conclude that $\mathcal{B}(\langle \pi_1,\pi_2\rangle)\subseteq B$. 
\end{proof}

\begin{theorem} \label{H2} The set $B$ is $\mathbb{F}_{q\sqrt{q}}$-linear.
\end{theorem}

\begin{proof} 
Lemma \ref{situatie'} shows that there exists a point $P$ of $B$, a tangent $(n-k)$-space $\pi$ at the point $P$ and at least $q^{3k-3}-q^{3k-4}-2\sqrt{q}q^{3k-5}$ $(n-k+1)$-spaces $H_i$  through $\pi$ for which $B\cap H_i$ is a Baer subplane, $i=1,\ldots,s$, $s\geq q^{3k-3}-q^{3k-4}-2\sqrt{q}q^{3k-5}$. Let $B\cap H_i=\mathcal{B}(\pi_i), i=1,\ldots, s$, with $\pi_i$ a plane.

Lemma \ref{ess'} shows that $\mathcal{B}(\langle \pi_i,\pi_j \rangle)\subseteq B$, $0\leq i\neq j\leq s$. 

If $k=2$, the set $\mathcal{B}(\langle \pi_1,\pi_2 \rangle)$ corresponds to a linear $2$-blocking set $B'$ in $\PG(n,q^3)$. Since $B$ is minimal, $B=B'$, and the Theorem is proven.

Let $k>2$. Denote the $(n-k+1)$-spaces trough $\pi$ different from $H_i$ by $K_j$, $j=1,\ldots, z$. There are at least $(q^{3k-3}-q^{3k-4}-2\sqrt{q}q^{3k-5}-1)/q^3$ different $(n-k+2)$-spaces $\langle H_1, H_j\rangle$, $1<j\leq s$. 
If all $(n-k+2)$-spaces $\langle H_1,H_j\rangle$, contain at least $2q^2$ of the spaces $K_i$, then $z\geq 2q^2(q^{3k-3}-q^{3k-4}-2\sqrt{q}q^{3k-5}-1)/q^3$, a contradiction if $q\geq 49$. Let $\langle H_1,H_2\rangle$ be an $(n-k+2)$-spaces containing less than $2q^2$ spaces $K_i$. 

Suppose, by induction, that for any $1<i<k$, there is an $(n-k+i)$-space $\langle H_1,H_2,\ldots, H_i\rangle$  containing at most $2q^{3i-4}$ of the spaces $K_i$, such that $\mathcal{B}(\langle \pi_1,\ldots,\pi_i\rangle)\subseteq B$. 

There are at least $\frac{q^{3k-3}-q^{3k-4}-2\sqrt{q}q^{3k-5}-(q^{3i}-1)/(q^3-1)}{q^{3i}}$ different $(n-k+i+1)$-spaces $\langle H_1,H_2,\ldots, H_i,H\rangle$, $H\not\subseteq \langle H_1,H_2,\ldots, H_i\rangle$. 

If all of these contain at least $2q^{3i-1}$ of the spaces $K_i$, then 
$$z\geq (2q^{3i-1}-2q^{3i-4})\frac{q^{3k-3}-q^{3k-4}-2\sqrt{q}q^{3k-5}-(q^{3i}-1)/(q^3-1)}{q^{3i}}+2q^{3i-4},$$ a contradiction if $q\geq 49$. Let $\langle H_1,\ldots, H_{i+1}\rangle$ be an $(n-k+i+1)$-space containing less than $2q^{3i-1}$ spaces $K_i$. We still need to prove that $\mathcal{B}(\pi_1,\ldots,\pi_{i+1})\subseteq B$.

 Since $\mathcal{B}(\langle \pi_{i+1}, \pi \rangle)\subseteq B$, with $\pi$ a plane in $ \langle \pi_1,\ldots,\pi_i\rangle$ for which $\mathcal{B}(\pi)$ is not contained in one of the spaces $K_i$, there are at most $2q^{3i-4}$ 4-dimensional spaces $\langle \pi_{i+1},\mu\rangle$ for which $\mathcal{B}(\langle\pi_{i+1},\mu\rangle)$ is not necessarily contained in $B$, giving rise to at most $2q^{3i-4}(q^6+q^4\sqrt{q})$ points $Q_i$ for which $\mathcal{B}(Q_i)$ is not necessarily in $B$.  Let $Q$ be a point of such a space $\langle\pi_{i+1},\mu\rangle$.
 
 There are $((q\sqrt{q})^{2i+2}-1)/(q\sqrt{q}-1)$ lines through $Q$ in $\langle \pi_1,\ldots, \pi_{i+1}\rangle \cong \PG(2i+2,q\sqrt{q})$, and there are at most $2q^{3i-4}(q^6+q^4\sqrt{q})$ points $Q_i$ for which $\mathcal{B}(Q_i)$ is not necessarily in $B$. Suppose all lines through $Q$ in $\langle \pi_1,\ldots, \pi_{i+1}\rangle \cong \PG(2i+2,q\sqrt{q})$ contain at least $q\sqrt{q}-q-\sqrt{q}$ points $Q_i$ for which $\mathcal{B}(Q_i)$ is not necessarily in $B$, then there are at least $(q\sqrt{q}-q-\sqrt{q}-1)((q\sqrt{q})^{2i+2}-1)/(q\sqrt{q}-1)+1>2q^{3i-4}(q^6+q^4\sqrt{q})$ points $Q_i$ for which $\mathcal{B}(Q_i)$ is not necessarily in $B$, a contradiction.

 Hence, there is a line $N$ through $Q$ in $\langle \pi_1,\ldots, \pi_{i+1}\rangle$ with at most $q\sqrt{q}-q-\sqrt{q}-1$ points $Q_i$ for which $\mathcal{B}(Q_i)$ is not necessarily contained in $B$, hence, for at least $q+\sqrt{q}+2$ points $R\in N$, $\mathcal{B}(R)\in B$. Result \ref{ba} yields that $\mathcal{B}(Q)\in B$.   This implies that $\mathcal{B}(\langle \pi_1,\ldots,\pi_{i+1}\rangle)\subseteq B$.
 
Hence, the space $\mathcal{B}(\langle H_1,H_2,\ldots, H_k\rangle)$ is such that $\mathcal{B}(\langle \pi_1,\ldots,\pi_k\rangle)\subseteq B$.  But $\mathcal{B}(\langle \pi_1,\ldots, \pi_k \rangle)$ corresponds to a linear $k$-blocking set $B'$ in $\PG(n,q^3)$. Since $B$ is minimal, $B=B'$.
\end{proof}

%%%%%%%%%%%%%%%%%%%%%%%%%%%%%%%%%%%%%%%%%%%%%%%%%%%

\end{document}